\documentclass[12pt,twoside,leqno]{article}
\usepackage[T1]{fontenc}
\usepackage{amsfonts}
\usepackage{amsmath,amsthm}
\usepackage{amssymb,latexsym}

\theoremstyle{plain}
\newtheorem{theorem}{Theorem}[section]
\newtheorem{lemma}[theorem]{Lemma}
\newtheorem{proposition}[theorem]{Proposition}
\newtheorem{corollary}[theorem]{Corollary}

\theoremstyle{definition}

\newtheorem{definition}[theorem]{Definition}

\newtheorem{problem}[theorem]{Problem}
\newtheorem{remark}[theorem]{Remark}

\DeclareMathOperator{\cl}{{\rm cl}}

\begin{document}
\title{A quasi-metrization theorem for hybrid topologies on the real line}
\author{Tom Richmond and Eliza Wajch\\
Department of Mathematics, \\Western Kentucky University\\Bowling Green, KY 42101, USA\\ email: tom.richmond@wku.edu\\
Institute of Mathematics, \\Siedlce University of Natural Sciences and Humanities,\\ 3. Maja 54, 08-110 Siedlce, Poland.\\
e-mail: eliza.wajch@gmail.com}

\maketitle
\begin{abstract}
Hybrid topologies on the real line have been studied by various authors.  Among the hybrid spaces, there are also Hattori spaces. However, some of the hybrid spaces are not homeomorphic to Hattori spaces. In this article, a common generalization of at least four kinds of the hybrid topologies on the real line is described. In the absence of the axiom of choice, a quasi-metrization theorem for such hybrid spaces is proved. It is shown that Kofner's quasi-metrization theorem for generalized ordered spaces is false in every model of $\mathbf{ZF}$ in which there exists an infinite Dedekind-finite subset of the real line.\medskip

\noindent\emph{Mathematics Suject Classification}: 54E35, 54A35, 54E15

\end{abstract}

\section{Introduction}
\label{Intro}

In this article, a 4-\emph{cover} of $\mathbb{R}$ is a collection $\mathcal{A}=\{A_1,A_2, A_3, A_4\}$ of pairwise disjoint subsets of $\mathbb{R}$ such that $\mathbb{R}=A_1\cup A_2\cup A_3\cup A_4$. Given a 4-cover $\mathcal{A}=\{A_1,A_2, A_3, A_4\}$ of $\mathbb{R}$, for $x\in\mathbb{R}$, we define the family $\mathcal{B}(x)$ as follows:
\[
\mathcal{B}(x)=\begin{cases} \{(x-\epsilon, x+\epsilon): \epsilon>0\} &\text{ if } x\in A_1;\\
\{\{x\}\} &\text{ if } x\in A_2;\\
\{[x, x+\epsilon): \epsilon>0\} &\text{ if } x\in A_3;\\
\{(x-\epsilon, x]: \epsilon>0\} &\text{ if } x\in A_4.\end{cases}
\]
There exists the unique topology $\tau_{\mathcal{A}}$ on $\mathbb{R}$ such that, for every $x\in\mathbb{R}$, the family $\mathcal{B}(x)$ is a local base at $x$ in $\langle \mathbb{R}, \tau_{\mathcal{A}}\rangle$. We call  $\tau_{\mathcal{A}}$ the \emph{hybrid topology determined by} $\mathcal{A}$. The topological space $H_4(\mathcal{A})=\langle\mathbb{R}, \mathcal{A}\rangle$ will be called the \emph{hybrid space determined by} $\mathcal{A}$. Similar general constructions, leading to the concept of a generalized ordered space, appeared for instance, in \cite{B}, \cite{Kof}, \cite{Lutz1} and \cite{Lutz}. Of course, $H_4(\mathcal{A})$ is a $T_3$-space and a generalized ordered space (see \cite{B}, \cite{Kof}, \cite{Lutz1}, \cite{Lutz}). We notice that if $A_2\cup A_3=\emptyset$, the space $H_4(\mathcal{A})=\langle\mathbb{R}, \mathcal{A}\rangle$ is the Hattori space $H(A_1)$. If $A_1=\mathbb{Q}$, $A_4=\mathbb{R}\setminus\mathbb{Q}$ (in this case, $A_2=A_3=\emptyset$), our $H_4(\mathcal{A})$ is the Engelking-Lutzer line (see \cite{EnL} and \cite[p.122]{Kof}).

We recall that Hattori spaces (called also $H$-spaces) were investigated, for instance, in \cite{BS}--\cite{ChH2} \cite{Hatt}, \cite{Kul}, \cite{Rich1}. If $A_1\cup A_2=\emptyset$, the space $J(A_3)$ in \cite{Rich1} is our $H_4(\mathcal{A})$. The space $J(\mathbb{R}\setminus\mathbb{Q})$ appeared in \cite[pp. 48, 166, 244,
281]{Rich2}. If $A_3\cup A_4=\emptyset$, the space $D_1(A_2)$ in \cite{Rich1} is our $H_4(\mathcal{A})$. If $A_1\cup A_4=\emptyset$, the space $D_2(A_2)$ in \cite{Rich1} is our $H_4(\mathcal{A})$. If $A_1\cup A_2\cup A_4=\emptyset$, the space $H_4(\mathcal{A})$ is the Sorgenfrey line, as usual, denoted by $\mathbb{S}$. If $A_1\cup A_2\cup A_3=\emptyset$, we denote $H(\mathcal{A})$ by $\mathbb{S}^{\leftarrow}$. Of course, the spaces $\mathbb{S}$ and $\mathbb{S}^{\leftarrow}$ are homeomorphic.

The set-theoretic framework for this work is the Zermelo-Fraenkel system $\mathbf{ZF}$. We recall that the axiom of choice $\mathbf{AC}$ is not an axiom of $\mathbf{ZF}$; in consequence, our reasoning in $\mathbf{ZF}$ may require more subtle arguments than in $\mathbf{ZFC}$ and, moreover, some known $\mathbf{ZFC}$-theorems on generalized ordered spaces cannot be applied to our work because they may fail in $\mathbf{ZF}$. 

The main aim of this article is to give a quasi-metrization theorem for hybrids of type $H_4(\mathcal{A})$ in the absence of the axiom of choice (see Theorems \ref{s2t5}, \ref{s2t8} and Corollary \ref{s2c09}). We also prove that Kofner's quasi-metrization theorem (see \cite[Theorem 10]{Kof}) is false if there exists an infinite Dedekind-finite subset of the real line (see Definition \ref{s2d11} and Theorem \ref{s2t11}). Finally, we give a sufficient condition for $H_4(\mathcal{A})$ to be metrizable in $\mathbf{ZF}$ (see Propositions \ref{s2p13}--\ref{s2p14}) and notice that $H_4(\mathcal{A})$ is both normal and completely regular in $\mathbf{ZF}$ (see Proposition  \ref{s2p15}).

Given a topological space $\mathbf{X}$ with the underlying set $X$, we denote by $\tau[\mathbf{X}]$ the topology of $\mathbf{X}$; hence, $\mathbf{X}=\langle X, \tau[\mathbf{X}]\rangle$. For a set $A\subseteq X$, $\cl_{\mathbf{X}}(A)$ and $\cl_{\tau[\mathbf{X}]}(A)$ denote the closure of $A$ in $\mathbf{X}$. 

Usually, if it is not stated otherwise, we denote topological spaces with boldface letters and their underlying sets with lightface letters. We may also denote the topology of $\mathbf{X}$ by $\tau$, writing $\mathbf{X}=\langle X, \tau\rangle$. 

In particular, the topology of the Sorgenfrey line is denoted by $\tau[\mathbb{S}]$. The natural topology of the real line $\mathbb{R}$, determined by the standard absolute value on $\mathbb{R}$, is denoted by $\tau[\mathbb{R}]$. For simplicity, if this is not confusing, we call the space $\langle\mathbb{R}, \tau[\mathbb{R}]\rangle$ simply the real line and denote it by $\mathbb{R}$.

All topological notions used in this paper, if not introduced here, are standard and can be found in \cite{En} or \cite{Will}. As usual, $\omega$ stands for the set of all finite ordinal numbers (of von Neumann). For our convenience, we denote by $\mathbb{N}$ the set $\omega\setminus\{0\}$ where $0=\emptyset$. For every ordinal $\alpha$, $\alpha+1=\alpha\cup\{\alpha\}$. We treat $\omega$ as the set of all non-negative integers in $\mathbb{R}$. The set of all rational numbers is denoted by $\mathbb{Q}$.

\section{The main results}

Let us recall some basic definitions to discuss conditions on $\mathcal{A}$ to get the quasi-metrizability of $H_4(\mathcal{A})$. 

\begin{definition}
\label{s2d1}
(Cf. \cite[pp. 3, 129]{FL}, \cite[Definition 1.10, p. 488]{Gr}, \cite{Wils}.)
A \emph{quasi-metric} (respectively, \emph{non-archimedean quasi-metric}) on a set $X$ is a function $\rho: X\times X\to [0, +\infty)$ satisfying the following conditions:
\begin{enumerate}
\item[(i)] $(\forall x, y\in X)(\rho(x,y)=0\leftrightarrow x=y)$;
\item[(ii)] $(\forall x,y,z\in X)\rho(x,y)\leq \rho(x,z)+\rho(z,y)$\\
 (respectively, $(\forall x,y,z\in X) \rho(x,y)\leq \max\{\rho(x,z), \rho(z,y)\}$).
\end{enumerate}
\end{definition}

Given a quasi-metric $\rho$ on a set $X$, a point $x\in X$ and a real number $r>0$, we put $B_{\rho}(x,r)=\{y\in X; \rho(x,y)<r\}$. We denote by $\tau(\rho)$ the topology on $X$ having the collection $\{B_{\rho}(x,r): x\in X\wedge r>0\}$ as a base. We call $\tau(\rho)$ the \emph{topology induced by} $\rho$. The quasi-metric $\rho$ is a metric if it satisfies the condition: $(\forall x,y\in X) d(x,y)=d(y,x)$.

\begin{definition}
\label{s2d2}
(Cf. \cite[pp. 3, 129]{FL}, \cite[Definition 10.1, p. 488]{Gr}, \cite{Wils}.)
\begin{enumerate}
\item[(i)] A topological space $\langle X, \tau\rangle$ is called (\emph{non-archimedeanly}) \emph{quasi-metri\-zable} if there exists a (non-archimedean) quasi-metric $\rho$ on $X$ such that $\tau=\tau(\rho)$.
\item[(ii)] If $\langle X,\tau \rangle$ is a quasi-metrizable space, then every quasi-metric $\rho$ on $X$ such that $\tau=\tau(\rho)$ is called a \emph{quasi-metric for} $\langle X, \tau\rangle$. 
\end{enumerate}
\end{definition}

Some authors use different definitions of a quasi-metric, not necessarily equivalent to Definition \ref{s2d1}. For instance, in \cite{Rich1}, condition (i) of Definition \ref{s2d1} is replaced with the following: $(\forall x,y\in X)(\rho(x,y)=\rho(y,x)=0\leftrightarrow x=y)$. However, in this paper, similarly to \cite{FL} and \cite{Gr}, we follow Wilson's terminology from \cite{Wils} and use the term ``quasi-metric'' only in the sense of Definition \ref{s2d1}.

\begin{definition}
\label{s2d3}
(Cf. \cite[pp. 29]{FL}, \cite[p. 489]{Gr}.)
 A family $\mathcal{U}$ of open subsets of a topological space $\mathbf{X}$ is called: 
\begin{enumerate}
\item[(i)] \emph{interior-preserving} if, for every non-empty subfamily $\mathcal{V}$ of $\mathcal{U}$, the intersection $\bigcap\mathcal{V}$ is open in $\mathbf{X}$;
\item[(ii)] $\sigma$-\emph{interior-preserving} if there exists a family $\{\mathcal{U}_n: n\in\omega\}$ such that $\mathcal{U}=\bigcup_{n\in\omega}\mathcal{U}_n$ and, for every $n\in\omega$, $\mathcal{U}_n$ is interior-preserving in $\mathbf{X}$.
\end{enumerate}
\end{definition}

Let us notice that the following well-known theorem is provable in $\mathbf{ZF}$:

\begin{theorem}
\label{s2t4}
(Cf. \cite[p. 151]{FL}, \cite[Theorem 10.3, p. 490]{Gr}.) $[\mathbf{ZF}]$ A $T_1$-space is non-archimedeanly quasi-metrizable if and only if it has a $\sigma$-interior-preserving base. 
\end{theorem}

In the sequel, we assume that  $\mathcal{A}=\{A_1,A_2, A_3, A_4\}$ is a given 4-cover of $\mathbb{R}$.

 Let us state our first main theorem on $H_4(\mathcal{A})$.

\begin{theorem}
\label{s2t5}
$[\mathbf{ZF}]$ 
Suppose that the 4-cover $\mathcal{A}$ satisfies the following conditions:
\begin{enumerate}
\item[(i)] there exists a family $\{F_n: n\in\omega\}$ of subsets of $A_3$ such that $A_3=\bigcup_{n\in\omega}F_n$ and, for every $n\in\omega$, $\cl_{\mathbb{S}^{\leftarrow}}(F_n)\subseteq A_2\cup A_3$;
\item[(ii)] there exists a family $\{H_n: n\in\omega\}$ of subsets of $A_4$ such that $A_4=\bigcup_{n\in\omega}H_n$ and, for every $n\in\omega$, $\cl_{\mathbb{S}}(H_n)\subseteq A_2\cup A_4$.
\end{enumerate}
 Then the space $H_4(\mathcal{A})$ is non-archimedeanly quasi-metrizable. In particular, if $A_3$ is of type $F_{\sigma}$ in $\mathbb{S}^{\leftarrow}$, and $A_4$ is of type $F_{\sigma}$ in $\mathbb{S}$, then $H(\mathcal{A})$ is non-archimedeanly quasi-metrizable.
\end{theorem}

\begin{proof}
In view of Theorem \ref{s2t4}, it suffices to show a $\sigma$-interior-preserving base for $H_4(\mathcal{A})$.

We fix families $\{F_n: n\in\omega\}$  and $\{H_n: n\in\omega\}$ which have the properties described in (i) and (ii), respectively.

Let $\mathcal{U}^{d}=\{\{x\}: x\in A_2\}$. Of course, $\mathcal{U}^{d}$ is interior-preserving in $H_4(\mathcal{A})$.

 Given $n\in\omega$ and $q\in\mathbb{Q}$, we put $\mathcal{U}_q(n)=\{[x, q): x\in F_n\wedge x<q\}$ and $\mathcal{V}_q(n)=\{(q, x]: x\in H_n\wedge q<x\}$.
 
  Let us show that the family $\mathcal{V}_q(n)$ is interior-preserving in $H_4(\mathcal{A})$. To this end, we fix a non-empty subfamily $\mathcal{W}$ of $\mathcal{V}_q(n)$. We may assume that $\bigcap\mathcal{W}\neq\emptyset$. We fix a point $x_0\in\bigcap\mathcal{W}$ and put $\mathcal{V}_q(n, x_0)=\{V\in\mathcal{V}_q(n): x_0\in V\}$. Since $x_0\in\mathcal{W}$, we have $\mathcal{V}_q(n, x_0)\neq\emptyset$. Let $V(x_0)=\bigcap\mathcal{V}_q(n, x_0)$. Of course, $x_0\in V(x_0)\subseteq\bigcap\mathcal{W}$. To prove that $\mathcal{V}_q(n)$ is interior-preserving in $H_4(\mathcal{A})$, it suffices to check that $V(x_0)$ is open in $H_4(\mathcal{A})$. 
  
Let $Y(x_0)=\{x\in H_n: x_0\leq x\}$. Since $\mathcal{V}_q(n, x_0)\neq\emptyset$, we infer that $Y(x_0)\neq\emptyset$. Let $y_0=\inf Y(x_0)$. Then $q<x_0\leq y_0$ and $(q, y_0]\subseteq V(x_0)\subseteq\bigcap\{(q, y]: y\in Y(x_0)\}$. This implies that $V(x_0)=(q, y_0]$. Since $\cl_{\mathbb{S}}(H_n)\subseteq A_4\cup A_2$, it follows that $(q, y_0]\in\tau[H_4(\mathcal{A})]$. This proves that all the families $\mathcal{V}_q(n)$ are interior-preserving in $H_4(\mathcal{A})$. Using similar arguments, one can show that, for every $q\in\mathbb{Q}$ and every $n\in\omega$, the family $\mathcal{U}_q(n)$ is also interior-preserving in $H_4(\mathcal{A})$. 
  
For every pair $p,q$ of rational numbers with $p<q$, we put $\mathcal{W}_{p,q}=\{(p, q)\}$. The family $\mathcal{W}_{p,q}$ is interior-preserving in $H_4(\mathcal{A})$. Furthermore, the union
\[\mathcal{U}^{d}\cup\bigcup\{\mathcal{W}_{p,q}: p,q\in\mathbb{Q}, p<q\}\cup\bigcup\{\mathcal{U}_q(n)\cup\mathcal{V}_q(n): q\in\mathbb{Q}\wedge n\in\omega\}\]

\noindent is a $\sigma$-interior-preserving base for $H_4(\mathcal{A})$.
\end{proof}

\begin{corollary}
\label{s2c6}
$[\mathbf{ZF}]$
Suppose that the 4-cover $\mathcal{A}$ is such that $A_3\cup A_4=\emptyset$. Then the hybrid space $H_4(\mathcal{A})$ is non-archimedeanly quasi-metrizable.
\end{corollary}

Under the assumptions of Corollary \ref{s2c6}, a quasi-metric (in general, not non-archimedean) for $H_4(\mathcal{A})$ is shown in \cite[Theorem 5.2]{Rich1}.

\begin{corollary}
\label{s2c7}
$[\mathbf{ZF}]$
Suppose that the 4-cover $\mathcal{A}$ is such that either $A_1\cup A_3=\emptyset$ or $A_1\cup A_4=\emptyset$.  Then the hybrid space $H_4(\mathcal{A})$ is non-archimedeanly quasi-metrizable.
\end{corollary}

Under the assumptions of our Corollary \ref{s2c7}, a quasi-metric for $H_4(\mathcal{A})=D_2(A_2)$ is shown in \cite[Theorem 5.2]{Rich1}.

One may try to deduce our Theorem \ref{s2t5} from \cite[Theorem 10]{Kof} or from \cite[Theorem 3.1]{B}; however, the results of \cite{B} and \cite{Kof} were proved in $\mathbf{ZFC}$ but not in $\mathbf{ZF}$,  no detailed $\mathbf{ZFC}$-proof of \cite[Theorem 10]{Kof} is included in \cite{Kof}, and the original proof of \cite[Theorem (3.1)]{B} in a more general setting is too complicated for $\mathbf{ZF}$-proofs of our theorems. Furthermore, we shall show later that \cite[Theorem 10]{Kof} is unprovable in $\mathbf{ZF}$. This is partly why we have given a simple, detailed proof of Theorem \ref{s2t5} and we give below a very clarified proof of Theorem \ref{s2t8} stating that the assumptions of Theorem \ref{s2t5} are also necessary for $H_4(\mathcal{A})$ to be quasi-metrizable in $\mathbf{ZF}$.

\begin{theorem}
\label{s2t8}
$[\mathbf{ZF}]$ If $H_4(\mathcal{A})$ is quasi-metrizable, then $\mathcal{A}$ satisfies conditions (i) and (ii) of Theorem \ref{s2t5}.
\end{theorem}
\begin{proof}
Let $\rho$ be a quasi-metric for $H_4(\mathcal{A})$. By using and clarifying some ideas from \cite[Section 3]{B}, we shall prove that condition (i) of Theorem \ref{s2t5} is satisfied. To this aim, we fix a bijection $q:\omega\to\mathbb{Q}$. For every $k\in\omega$, let $Q_k=\{q(i): i\in k+1\}$. For every $n\in\omega$ and every $x\in\mathbb{R}$, let $E_n(x)=B_{\rho}(x, \frac{1}{2^n})$. For $x\in A_3$ and  $n\in\omega$, let $m(x,n)=\min\{i\in\omega: [x, x+\frac{1}{2^i})\subseteq E_n(x)\}$, $j(x,n)=\max\{n, m(x,n)\}$ and $D_n(x)=[x,x+\frac{1}{2^{j(x,n)}})$. For every pair $k,n\in\omega$, we define a subset $F_{k,n}$ of $A_3$ as follows:
\[F_{k,n}=\{x\in A_3: E_n(x)\subseteq [x, +\infty)\wedge D_{n+1}(x)\cap (Q_k\setminus\{x\})\neq\emptyset\}.\]

It is obvious that $A_3=\bigcup\{F_{k,n}: k,n\in\omega\}$. Let us fix $k,n\in\omega$ and show that $\cl_{\mathbb{S}^{\leftarrow}}(F_{k,n})\subseteq A_2\cup A_3$. 

Let $p\in\cl_{\mathbb{S}^{\leftarrow}}(F_{k,n})$. Suppose that $p\in A_1\cup A_4$. Let $q_0=\max{Q_k}$. Suppose that $q_0< p$. Then there exists $x_0\in F_{k,n}\cap(q_0, p]$. But this is impossible because $D_{n+1}(x_0)\cap Q_k\neq\emptyset$. The contradiction obtained shows that $p\leq q_0$. Let $q_p=\min\{q(i): i\in k+1\wedge p\leq q(i)\}$. We can choose $\epsilon>0$ such that $(p-\epsilon, p)\cap Q_k=\emptyset$. We notice that if $x\in F_{k,n}\cap (p-\epsilon, p)$, then $q_p\in D_{n+1}(x)$, so $p\in D_{n+1}(x)$. Since  $p\in\cl_{\mathbb{S}^{\leftarrow}}(F_{k,n})$ and $p\in A_1\cup A_4$, we deduce that there exist $x_1,x_2\in F_{k,n}\cap E_{n+1}(p)\cap (p-\epsilon, p)$ such that $x_1<x_2$. Since $p\in D_{n+1}(x_2)\subseteq E_{n+1}(x_2)$, we infer that $E_{n+1}(p)\subseteq E_n(x_2)$. Hence $x_1\in E_n(x_2)\subseteq [x_2, +\infty)$. This is impossible. The contradiction obtained proves that $p\in A_0\cup A_3$. Hence (i) of Theorem \ref{s2t5} is satisfied. Arguing similarly, one can show that also condition (ii) of Theorem \ref{s2t5} is satisfied.
\end{proof}

\begin{corollary}
\label{s2c09}
$[\mathbf{ZF}]$
For every 4-cover $\mathcal{A}$ of $\mathbb{R}$, the following conditions are all equivalent:
\begin{enumerate}
\item[(i)] $H_4(\mathcal{A})$ is non-archimedeanly quasi-metrizable;
\item[(ii)] $H_4(\mathcal{A})$ is quasi-metrizable;
\item[(iii)] $\mathcal{A}$ satisfies conditions (i)--(ii) of Theorem \ref{s2t5}.
\end{enumerate}
\end{corollary}

\begin{corollary}
\label{s2c9}
$[\mathbf{ZF}]$ Suppose that the 4-cover $\mathcal{A}$ is such that $A_2=\emptyset$. Then the space $H_4(\mathcal{A})$ is quasi-metrizable if and only if $A_3$ is of type $F_{\sigma}$ in $\mathbb{S}^{\leftarrow}$, and $A_4$ is of type $F_{\sigma}$ in $\mathbb{S}$.
\end{corollary}

\begin{problem}
\label{s2p10}
Find, if possible, a subset $F$ of $\mathbb{R}$ such that $F$ is of type $F_{\sigma}$ in $\mathbb{S}$ but $F$ is not of type $F_{\sigma}$ in $\mathbb{R}$.
\end{problem}

To show our next corollary and that Theorem 10 of \cite{Kof} is unprovable in $\mathbf{ZF}$, we need the following lemma. We include its $\mathbf{ZF}$-proof for completeness.

\begin{lemma}
\label{s2l12}
$[\mathbf{ZF}]$
If a subset $F$ of $\mathbb{R}$ is closed either in $\mathbb{S}$ or in $\mathbb{S}^{\leftarrow}$, then $F$ is of type $G_{\delta}$ in $\mathbb{R}$. 
\end{lemma}
\begin{proof}
Suppose that $F\subseteq\mathbb{R}$ is closed in $\mathbb{S}^{\leftarrow}$. We put $H=\cl_{\mathbb{R}}(F)$ and $C=H\setminus F$,  Since $F$ is closed in $\mathbb{S}^{\leftarrow}$, for every $x\in C$, the set $J(x)=\{j\in\mathbb{N}: (x-\frac{1}{j},x]\cap F=\emptyset\}$ is non-empty, so we can define $j(x)=\min J(x)$. Since the family $\{(x-\frac{1}{j(x)}, x]: x\in C\}$ is disjoint, it is countable. Therefore, the set $C$ is countable. This implies that $\mathbb{R}\setminus C$ is of type $G_{\delta}$ in $\mathbb{R}$. Since $H$ is also of type $G_{\delta}$ in $\mathbb{R}$, the set $F= H\cap(\mathbb{R}\setminus C)$ is of type $G_{\delta}$ in $\mathbb{R}$. To complete the proof, we consider the function $f:\mathbb{R}\to\mathbb{R}$ defined by: for every $x\in\mathbb{R}$, $f(x)=-x$. Since $f$ is a homeomorphism of $\mathbb{S}$ onto $\mathbb{S}^{\leftarrow}$, we infer that every closed set of $\mathbb{S}$ is also of type $G_{\delta}$ in $\mathbb{R}$.
\end{proof}

\begin{corollary}
\label{s2c013}
$[\mathbf{ZF}]$
Assume that the 4-cover $\mathcal{A}$ is such that $A_2$ is of type $F_{\sigma}$ in $\mathbb{R}$ but either $A_3$ or $A_4$ is not of type $G_{\delta \sigma}$ in $\mathbb{R}$. Then the hybrid space $H_4(\mathcal{A})$ is not quasi-metrizable.
\end{corollary}
\begin{proof}
Suppose the space $H_4(\mathcal{A})$ is quasi-metrizable. It follows from Theorem \ref{s2t8} that there exists a family $\{F_n: n\in\omega\}$ of subsets of $A_3$ such that $A_3=\bigcup_{n\in\omega}F_n$ and, for every $n\in\omega$, $\cl_{\mathbb{S}^{\leftarrow}}(F_n)\subseteq A_3\cup A_2$. We fix $n\in\omega$ and put $H_n=\cl_{\mathbb{S}^{\leftarrow}}(F_n)$. By Lemma \ref{s2l12}, the set $H_n$ is of type $G_{\delta}$ in $\mathbb{R}$. Since $A_2$ is of type $F_{\sigma}$ in $\mathbb{R}$, the set $E_n=H_n\setminus A_2$ is of type $G_{\delta}$ in $\mathbb{R}$. Clearly, $F_n\subseteq E_n\subseteq A_3$. Hence $A_3=\bigcup_{n\in\omega}E_n$ is of type $G_{\delta\sigma}$ in $\mathbb{R}$. Using similar arguments, one can also prove that $A_4$ is of type $G_{\delta\sigma}$ in $\mathbb{R}$. This completes the proof.
\end{proof}

We are going to show below that Kofner's quasi-metrization theorem (see \cite[Theorem 10]{Kof}) may fail in $\mathbf{ZF}$.  To do this, in the following definition, we translate condition (ii) of \cite[Theorem 10]{Kof} into our language for Hattori spaces.

\begin{definition}
\label{s2d11}
\emph{Kofner's quasi-metrization theorem for Hattori spaces} is the following statement:  For every subset $A$ of $\mathbb{R}$, the Hattori space $H(A)$ is non-archimedeanly quasi-metrizable if and only if there exists a family $\{R_n: n\in\omega\}$ of subsets of $\mathbb{R}$ such that $\mathbb{R}\setminus A=\bigcup_{n\in\omega}R_n$ and, for every $n\in\omega$, $R_n$ is closed under limits in $\mathbb{R}$ of increasing sequences. (See \cite[$(ii)\leftrightarrow (iii)$ of Theorem 10]{Kof}.)
\end{definition}

\begin{remark}
\label{s2r12}
Let $A\subseteq \mathbb{R}$. It follows from Corollary \ref{s2c9} that the following theorem is true in $\mathbf{ZF}$: if the Hattori space $H(A)$ is quasi-metrizable, then there exists a family $\{R_n: n\in\omega\}$ such that $\mathbb{R}\setminus A=\bigcup_{n\in\omega}R_n$ and, for every $n\in\omega$, $R_n$ is closed under limits of increasing sequences. That the converse may fail in $\mathbf{ZF}$ is shown in the forthcoming Theorem \ref{s2t11}.
\end{remark}

We recall that a set $X$ is called \emph{Dedekind-infinite} if there exists an injection $\psi: \omega\to X$. If a set is not Dedekind-infinite, it is called \emph{Dedekind-finite}. Basic properties of infinite Dedekind-finite subsets of $\mathbb{R}$ are given in \cite{Herr}. 

\begin{theorem}
\label{s2t11}
$[\mathbf{ZF}]$
If $D$ is an infinite Dedekind-finite subset of $\mathbb{R}$,  then the Hattori space $H(\mathbb{R}\setminus D)$ is not quasi-metrizable. In consequence, Kofner's quasi-metrization theorem is false in every model of $\mathbf{ZF}$ in which $\mathbb{R}$ contains an infinite Dedekind-finite set. In particular, Kofner's theorem is false in Cohen's original model $\mathcal{M}1$ in \cite{HR1}.
\end{theorem}

\begin{proof}
Let $D$ be an infinite Dedekind-finite subset of $\mathbb{R}$. Put $A=\mathbb{R}\setminus D$ and suppose that the Hattori space $H(A)$ is quasi-metrizable. Then it follows from Corollary \ref{s2c013} that the set $D$ is of type $G_{\delta\sigma}$ in $\mathbb{R}$. Thus, there exists a family $\{G_n: n\in\omega\}$ of $G_{\delta}$-sets in $\mathbb{R}$ such that $D=\bigcup_{n\in\omega}G_n$. Corollary 2.11 of \cite{KW1} states, among other things, that if an infinite subset $Y$ of a Cantor completely metrizable, second-countable space $\mathbf{X}$ is of type $G_{\delta}$ in $\mathbf{X}$, then $Y$ is Dedekind-infinite. Therefore, since $\mathbb{R}$ is a Cantor completely metrizable, second-countable space, and, for every $n\in\omega$, the set $G_n$ is Dedekind-finite and of type $G_{\delta}$ in $\mathbb{R}$, it follows that, for every $n\in\omega$, the set $G_n$ is finite. In $\mathbf{ZF}$, countable unions of finite subsets of $\mathbb{R}$ are countable. This implies that the set $\bigcup_{n\in\omega}G_n$ is countable. But this is impossible because $D$ is uncountable. The contradiction obtained shows that $H(A)$ is not quasi-metrizable. Since $D$ contains no infinite countable subsets, it is closed under limits of increasing sequences. Therefore, Kofner's quasi-metrization theorem is false in every model of $\mathbf{ZF}$ in which there is an infinite Dedekind-finite subset of $\mathbb{R}$.
\end{proof}

It seems that a hard problem is to find useful necessary and sufficient conditions for $\mathcal{A}$ to determine a metrizable $H_4(\mathcal{A})$. We recall that, in $\mathbf{ZF}$, the statement ``Every metrizable space has a $\sigma$-locally finite base'' is unprobvable (see \cite[Form 232]{HR1}, \cite{HKRS}, \cite{CHHKR}). However, it was observed, e.g., in \cite{CHHKR} that it follows from \cite[the proof of Theorem 7]{CR} that the following part of Nagata-Smirnov metrization theorem is true in $\mathbf{ZF}$:

\begin{theorem}
\label{s2t12}
$[\mathbf{ZF}]$ 
If a  $T_3$-space has a $\sigma$-locally finite base, it is metrizable.
\end{theorem}

We can state the following simple proposition:

\begin{proposition}
\label{s2p13}
$[\mathbf{ZF}]$
If the 4-cover $\mathcal{A}$ is such that $A_2$ is of type $F_{\sigma}$ in $\mathbb{R}$, and $A_3\cup A_4$ is countable, then the hybrid space $H_4(\mathcal{A})$ is metrizable.
\end{proposition}
\begin{proof}
Under the assumptions of the proposition, the space $H_4(\mathcal{A})$ has a $\sigma$-locally finite base, so it is metrizable by Theorem \ref{s2t12}.
\end{proof}

\begin{proposition}
\label{s2p14}
$[\mathbf{ZF}]$
Suppose that the 4-cover $\mathcal{A}$ is such that $A_2$ is countable. Then the following conditions are equivalent:
\begin{enumerate}
\item[(i)] $H_4(\mathcal{A})$ is second-countable;
\item[(ii)] $H_4(\mathcal{A})$ is metrizable;
\item[(iii)] $A_3\cup A_4$ is countable.
\end{enumerate}  
\end{proposition}
\begin{proof}
Since $A_2$ is countable, the space $H_4(\mathcal{A})$ is separable. Hence, if $H_4(\mathcal{A})$ is metrizable, it is also second-countable. Clearly, a subspace of the Sorgenfrey line is second-countable if and only if it is countable. Every second-countable $T_3$-space is metrizable by Theorem \ref{s2t12}, and, moreover, $A_2\cup A_3\cup A_4$ is countable if and only if $H_4(\mathcal{A})$ is second-countable. All this taken together completes the proof.
\end{proof}

Finally, we recall that even a linearly ordered normal space can fail to be completely regular in $\mathbf{ZF}$ (see, e.g., \cite{GT}), and a linearly orderable topological space may fail to be normal in $\mathbf{ZF}$ (see \cite[Form 118, p. 42]{HR1}); therefore, it makes sense to establish the following result:

\begin{proposition}
\label{s2p15}
$[\mathbf{ZF}]$ For every 4-cover $\mathcal{A}$ of $\mathbb{R}$, the space $H_4(\mathcal{A})$ is both normal and completely regular.
\end{proposition}
\begin{proof}
Let $C_0, C_1$ be a pair of non-empty disjoint closed sets of $H_4(\mathcal{A})$.  Let $i\in\{0,1\}$ and $c\in C_i$. If  $c\in A_2$, we put $U_i(c)=\{c\}$. Assuming that  $c\in C_i\setminus  A_2$, we define the set $N(c)$ as follows:
\[
N(c)=\begin{cases} \{n\in\mathbb{N}: (c-\frac{1}{n}, c+\frac{1}{n})\cap C_{1-i}=\emptyset\} &\text{ if } c\in A_1;\\
\{n\in\mathbb{N}: [c, c+\frac{1}{n})\cap C_{1-i}=\emptyset\} &\text{ if } c\in A_3;\\
\{n\in\mathbb{N}: (c-\frac{1}{n}, c]\cap C_{1-i}=\emptyset\} &\text{ if } c\in A_4.\end{cases}\]
\noindent We put $n(c)=\min N(c)$ and define the set $U_i(c)$ as follows:
\[
U_i(c)=\begin{cases} (c-\frac{1}{2n(c)}, c+\frac{1}{2n(c)}) &\text{ if } c\in A_1;\\
 [c, c+\frac{1}{2n(c)}) &\text{ if } c\in A_3;\\
(c-\frac{1}{2n(c)}, c] &\text{ if } c\in A_4.\end{cases}\]

\noindent We define $V_i=\bigcup\{U_i(c): c\in C_i\}$. Clearly, $V_i\in\tau[H_4(\mathcal{A})]$ and $C_i\subseteq V_i$. Arguing in much the same way as in the proof of Theorem 2.3 in \cite{Rich1}, one can check that $V_0\cap V_1=\emptyset$. This proves that the space $H_4(\mathcal{A})$ is normal.

To show that the space $H_4(\mathcal{A})$ is completely regular, we consider any closed set $E$ in $H_4(\mathcal{A})$ and any point $x\in \mathbb{R}\setminus E$. We choose $U\in\mathcal{B}(x)$ such that $U\cap E=\emptyset$. Suppose that $x\in A_3$. Then, for some $\epsilon>0$, $U=[x, x+\epsilon)$. We define the function $f_x: \mathbb{R}\to [0,1]$ as follows:
\[
f_x(t)=\begin{cases} 1 &\text{ if } t\in (-\infty, x)\cup [x+\epsilon, +\infty);\\
0 &\text{ if } t\in [x, x+\frac{\epsilon}{2}];\\
\frac{2}{\epsilon}(t-x)-1 &\text{ if } t\in (x+\frac{\epsilon}{2}, x+\epsilon).\end{cases}\]
Then $f_x$ is a continuos function from $H_2(\mathcal{A})$ into $[0,1]$ such that $f_x(x)=0$ and $E\subseteq f_x^{-1}[\{1\}]$. Using similar arguments, one can show that if $x\in A_1\cup A_2\cup A_4$, then there exists a continuous function $f_x$ from $H_2(\mathcal{A})$ into $[0,1]$ such that $f_x(x)=0$ and $E\subseteq f_x^{-1}[\{1\}]$. Hence, $H_4(\mathcal{A})$ is completely regular.
\end{proof}

We would like to encourage readers to investigate other topological properties of the hybrid spaces $H_4(\mathcal{A})$ in $\mathbf{ZF}$.

\end{document}